\documentclass{RSMUP}
\address[rodolphe.richard@normalesup.org]{Rodolphe \textsc{RICHARD}}

\newif\ifXELATEX
\XELATEXfalse

\usepackage[english,frenchb]{babel}
\usepackage{mathrsfs}
\usepackage{pdfpages}
\ifXELATEX
	\usepackage{fontspec} 
	\defaultfontfeatures{Mapping=tex-text} 
	\usepackage{xunicode} 
	\usepackage{xltxtra} 
\else 
	\usepackage[T1]{fontenc} 
	\usepackage[utf8]{inputenc} 
\fi

\usepackage[colorlinks=true,citecolor=blue]{hyperref}

\received[\today]{\today}

\newtheorem{theorem}{Theorem}
\newtheorem{corollary}[theorem]{Corollary}
\newtheorem{lemma}[theorem]{Lemma}
\newtheorem{proposition}[theorem]{Proposition}

\theoremstyle{definition}

\newtheorem{example}[theorem]{Example}
\newtheorem{remark}[theorem]{Remark}

\newcommand{\QQ}{{\mathbf{Q}}}
\newcommand{\eps}{\varepsilon}
\newcommand{\abs}[1]{{\left|{#1}\right|}}
\newcommand{\Nm}[1]{{\left\|{#1}\right\|}}
\newcommand{\ceil}[1]{{\left\lceil{#1}\right\rceil}}
\newcommand{\floor}[1]{{\left\lfloor{#1}\right\rfloor}}
\newcommand{\pgcd}[2]{\mathrm{pgcd}\left(#1,#2\right)}

\newcommand{\RoC}{\texttt{RoC}}

\title[On~$\pi$-exponentials~II]{On~$\pi$-exponentials~II:\\
Closed formula for the index
}

\author[Rodolphe Richard]{Rodolphe Richard}

\makeatletter
\@addtoreset{paragraph}{section}
\renewcommand{\theparagraph}{\textbf{\@arabic\c@section.\@arabic\c@paragraph}}
\renewcommand{\thesubparagraph}{\textbf{\@arabic\c@section.\@arabic\c@paragraph.\@arabic\c@subparagraph}}
\newcommand{\Odagger}{
\mathcal{O}^\dagger
}
\makeatother

\begin{document}

\maketitle

\selectlanguage{french}
\begin{abstract} Cet article poursuit la série, entamée avec~\cite{piexp1}, dédiée aux~$\pi$-exponentielles de Pulita
et aux équations différentielles $p$-adique de rang $1$ polynomiales dans une extension ultramétrique du corps des
nombres $p$-adiques. Nous complétons~\cite{piexp1} avec une formule close pour l’indice. Le cas particulier «~$p$-typique~» 
résoud un problème étudié dans~\cite{Morofushi}. Nous répondons également à une question~\cite[§2.4]{Robba} de Robba
sur la comparaison de la cohmologie rationnelle vers celle de Dwork. Nous indiquons même une procédure pour allier aux cas
où il n’y a pas isomorphisme. Nous établissons en passant une caractérisation computationelle des équations solubles à équivalence 
près sur l’algèbre dague. Un appendice détermine la complexité polynomiale de l’algorithme dérivé.
\end{abstract}

\selectlanguage{english}
\begin{abstract} This article pursue the series, initiated by~\cite{piexp1}, dedicated to Pulita’s~$\pi$-exponentials and
$p$-adic differential equation of rank one with coefficients a polynomial in a ultrametric extension of the field of $p$-adic
numbers.
We complement~\cite{piexp1} with a closed formula for the index. The “$p$-typical” particular case answers one problem
studied  in~\cite{Morofushi}. We also answer a question~\cite[§2.4]{Robba} of Robba on the comparison from rational cohomology
toward Dwork cohomology. We eve indicate a procedure to palliate the lack of isomorphy of this comparison. We establish by 
the way a characterisation of soluble equations up to equivalence on the dagger algebra. An appendix determine the polynomial 
complexity of the
bderived algorithm.
\end{abstract}

\setcounter{secnumdepth}{5}
\setcounter{tocdepth}{1}



\tableofcontents

\paragraph{General notations.}
Fix a prime~$p$ and a ultrametric field extension~$K$ of~$\QQ_p$, and write~$\abs{-}$ for its absolute value. 
 An element~$x$ of~$K$ is a (ultrametric) \emph{integer} of~$K$
if~$\abs{x}\leq1$. We denote~$R$ the ring of integers of~$K$, and~$\kappa$ its residue field.
$$\text{(characteristic $0$)}\qquad K \hookleftarrow R \twoheadrightarrow \kappa\qquad\text{(characteristic $p$)}$$

\section{Results} 
\paragraph{Problem.}
For any~$L(T)$ in~$K[T]$ we consider the differential equation
\begin{equation}\label{diffeq}
y^\prime=L(T)·y 
\end{equation}
Let~$P(T)$ be given by:~$P(0)=0$ and~$P^\prime(T)=L(T)$, so that the series~
\begin{equation}\label{formalsol}
e(T)=\exp(P(T))
\end{equation}
is defined  in~$K[[T]]$ and is a solution of~$\eqref{diffeq}$. 

We may refer indifferently to a differential equation like~\eqref{diffeq} through the 
equation~\eqref{diffeq} itself, the corresponding polynomial~$P(T)$, or the corresponding series~$e(T)$.

\subparagraph{}
Denote~$\Odagger\subseteq K[[T]]$ the sub-algebra of “overconvergent” series: convergent series
with radius of convergence~$>1$. We consider~\eqref{diffeq} as a differential equation
over~$\Odagger$.
 One says that~\eqref{diffeq}
\begin{itemize}
\item is~\emph{trivial} if~$e(T)$ has radius~$>1$, and, more generally,
\item is~\emph{soluble} if~$e(T)$ has radius~$\geq 1$ (radius~$1$ is included.)
\end{itemize}
\subparagraph{}\label{paraequiv}
 Given two differential equations such as~\eqref{diffeq}, with corresponding series~$e_1(T)$ and~$e_2(T)$,
they are \emph{equivalent} if the identity~$e_1\cdot\Odagger =e_1\cdot\Odagger $ holds. Equivalently, the series~$e_1(T)/e_2(T)$,
(which corresponds to the difference of the equations,)  is in~$\Odagger$: its radius of convergence is~$>1$.
\subparagraph{}
We are concerned here with the computation of numerical invariant, under equivalence~\ref{paraequiv}, associated with soluble~\eqref{diffeq}: its \emph{index}~$\chi\in~\mathbf{Z}$.
(cf.~§\ref{sectionindex}) Equivalent invariants are the \emph{$p$-adic irregularity} of~\cite[§2.4]{Pulita}, the \emph{first slope} in~\cite[Definition~2.6]{Christol}, the \emph{Swan conductor} (\cite[Théorème~1.4.4.~4.]{Pulita}.)
\begin{remark} In principle, the determination of the radius of convergence function from~\cite[§5~(23)]{piexp1} allows us to infer quite directly
the slopes and then the index. Nevertheless, formula~\cite[§5~(23)]{piexp1} is computationally more involved than~\cite[Théorème~3]{piexp1}. We will obtain here a computationally more direct approach, yielding a more satisfying answer regarding the applications~\ref{appliexpsums} and~\ref{applicomparison}.
\end{remark}
\paragraph{Some notations from~\cite{piexp1}.} Fix an integer~$D\geq\deg(P)$ and write
\begin{equation}\label{indices}
 d=\floor{\log_p(D)}\text{, and }d_i=\floor{\log_p(D/i)}\text{ for }1\leq i\leq D.
\end{equation}
We assume that~$K$ has a primitive root of unity~$\zeta$ of order~$p^{d+1}$,
and denote
\begin{equation}\label{uniformizers}
\pi_i=\zeta^{p^{d-i}}-1\text{ for }1\leq i\leq d.
\end{equation}
(uniformisers of a tower~$\QQ_p(\pi_0)\subseteq\ldots\subseteq\QQ_p(\pi_d)$
of ramified cyclotomic extensions.)

Write~$P(T)$ as~$\sum_{i=1}^{D} a_i\cdot T^i$, let
\begin{equation}\label{Formulatilde}
\begin{split}
\tilde{P}(T)
& = \sum_{i=1}^{D} \left.a_i\cdot T^i\right/\pi_i \\
\tilde{e}(T)
& =  \exp(\tilde{P}(T))\pmod{\left(T^{D+1}\right)}.
\end{split}
\end{equation}
\begin{remark} The integers~\eqref{indices} are the ones which describe the decomposition of the ring of truncated \emph{universal} Witt vectors of length~$D$ into products of rings of~$p$-typical Witt vectors, of lengths the~$d_i$. (cf.~\cite[§2.6]{piexp1} and~§\ref{ptypdecomp}) The uniformisers~$\pi_i$, and more general ones, comes from the work~\cite{Pulita} of Pulita. These ones were already found in~\cite{Matsuda}. The appendix~\cite[§C]{piexp1} applies here: everything proceeds without modification with the more general~$\pi_i$ of Pulita. 

\end{remark}

\paragraph{Some results of~\cite{piexp1}.} This gathers what we need from~\cite{piexp1}.

\begin{theorem}[{\cite{piexp1}}]\label{thmarticle}
The following are equivalent (with notations above):
\begin{enumerate}
\item the radius of convergence is~$\geq 1$ (resp.~$>1$);\label{condition1}
\item the coefficients of~$e(T)$ are integers (resp. are eventually divisible by~$\pi_0$); \label{condition2}
\item the coefficients of~$\tilde{e}(T)$ are integers (resp.~$\tilde{e}(T)$ reduces to~$1$ in~$\kappa[T]/(T^{D+1})$.) \label{condition3}
\end{enumerate}
\end{theorem}
\begin{remark}\label{remark}
Recall this correspond to the solubility (resp. triviality) of~\eqref{diffeq}.
We note that, in condition~\eqref{condition3}, the reduction invoked in the trivial case is meaningful thanks to the integrality expressed in the
solvable case. 
\end{remark}
\begin{proof}
 The solubility case of Theorem~\ref{thmarticle} is namely~\cite[§2.5 Théorème~2, §2.10~Corollaire~1]{piexp1}.
The equivalence of the first two statements in the triviality case follows from~\cite[Proposition~4]{piexp1}. The equivalence
of the first and third statement in the triviality case follows from the formula for the radius of convergence~\cite[Théorème~3]{piexp1}.
\end{proof}

\paragraph{Characterisation.} In the solvable case, cf Remark~\ref{remark}, we may reduce~$\tilde{e}(T)$ into
\begin{equation}\label{eHat}
\widehat{e}(T) \in \kappa[T]/(T^{D+1}).
\end{equation}
As a consequence of Theorem~\ref{thmarticle}, a soluble~\eqref{diffeq} is characterised by~$\widehat{e}(T)$ as follows.
\begin{proposition}[Characterisation of differential equations]\label{equivalence}
 Consider 
\begin{itemize}
\item 	two polynomials~$L_1(T)$ and~$L_2(T)$ in~$K[T]$, each of degree at most~$D$;
\item    the corresponding differential equations, say~\eqref{diffeq}$_1$ and~\eqref{diffeq}$_2$ resp.;
\item    and the corresponding truncated series~$\widehat{e}_1(T)$ and~$\widehat{e}_2(T)$.
\end{itemize}
Assume solubility of~\eqref{diffeq}$_1$ or~\eqref{diffeq}$_2$. Then~\eqref{diffeq}$_1$ and~\eqref{diffeq}$_2$ are equivalent if and only
if
\begin{equation}\widehat{e}_1(T)=\widehat{e}_2(T).\label{propformula}\end{equation}	
\end{proposition}
As may be expected we will extract the index form this complete invariant~$\widehat{e}(T)$.
\begin{proof} Solubility is invariant under equivalence; we can assume both~\eqref{diffeq}$_1$ and~\eqref{diffeq}$_2$ are soluble. It suffices to show the equation associated with~$L_1(T)-L_2(T)$ is trivial. The associated truncated series is~$\tilde{e}_1(T)/\tilde{e}_2(T)\pmod{\left( T^{D+1}\right)}$. 
It has integral coefficients (recall~$1+TR[[T]]/(T^{D+1})$ is a multiplicative group.) The identity\eqref{propformula} is equivalent to the triviality of the reduction of~$\tilde{e}_1(T)/\tilde{e}_2(T)$ in $\kappa[[T]]/{\left( T^{D+1}\right)}$. By Theorem~\ref{thmarticle}, this is equivalent to condition~\eqref{condition1} of Theorem~\ref{thmarticle}. This concludes.
\end{proof}

\begin{remark}  Conversely, we may lift a given~$\widehat{e}(T)$ to some~$\tilde{e}(T)$ in~$R[T]/(T^D+1)$, write~$\tilde{P}(T)$ the logarithm of 
the latter, considered as a polynomial, deduce~$P(T)$, take its derivative~$L(T)$ and get an equation~\eqref{diffeq} which will produce this~$\widehat{e}(T)$.
\end{remark}
\begin{remark}\label{examples} The construction of~$\widehat{e}(T)$ from~\eqref{diffeq} depends on the choice of~$D$. For example, take~$L(T)$ 
is the constant polynomial~$\pi_0$, so that~$e(T)=\exp(\pi_0\cdot T)$, known to have radius~$1$. 
\begin{enumerate}
\item \label{example1} For~$D=1$, one gets~$\tilde{P}(T)=T$, and one has~$$\tilde{e}(T)=1+T\pmod{(T^2)}$$ 
and~$\widehat{e}(T)=1+T\pmod{(T^2)}$.
\item \label{example2} For~$D=p-1$, one gets~$\tilde{P}(T)=T$, and  one has~$$\tilde{e}(T)=1+T+\ldots+T^p/(p-1)!\pmod{(T^p)}$$ and~$\widehat{e}(T)=1+T+\ldots+T^{p-1}/(p-1)!\pmod{(T^p)}$.
\item \label{example3} For~$D=p$, one gets~$\tilde{P}(T)=\frac{\pi_0}{\pi_1}T$, and  one has~
$$\tilde{e}(T)=1+\frac{\pi_0}{\pi_1}T+\ldots+\left(\frac{\pi_0}{\pi_1}T\right)^{p-1}\cdot\frac{1}{(p-1)!}+\left(\frac{\pi_0}{\pi_1}T\right)^{p}\cdot\frac{1}{p!}\pmod{(T^{p+1})}$$
 and~$\widehat{e}(T)=1+0+u\cdot T^{p}\pmod{(T^p)}$ for the unit~$u=-1$ of~$\mathbf{Z}/(p)$.
\end{enumerate} 
\end{remark}

\begin{remark}[cf.{\cite[Intoduction]{Pulita}}] A famous landmark result is the $p$-adic local monodromy theorem (formerly Crew’s conjecture). The rank one case gets a 
correspondence between Artin-Schreier-Witt characters of absolute Galois group of a local field in characteristic zero and rank one differential equations over the Robba ring. Pulita’s $\pi$-exponential was developed in order to made this correspondence explicit. Given a solvable differential equation~\eqref{diffeq}, and some~$D$, we constructed a complete invariant~$\widehat{e}(T)$. This is a truncated power series, but can be interpreted equivalently as a truncated universal Witt vector, by first theorem of Cartier theory of Witt vectors. Witt motivation for Witt vectors was the classification of cyclic coverings on characteristic zero. The universal Witt vector obtained has a $p$-typical decomposition, and each factor corresponds to some Artin-Schreier-Witt covering, composed with a Kummer covering. Each of this coverings generates by relative rigid cohomology, a degree~$1$ “$F$-isocrystal” over the affine line in characteristic~$p$ with action of~$\mathbf{Z}/p^{d}\mathbf{Z}$ for some~$d$. The choice of~$\zeta$ determines a character~$\chi_\zeta$ of~$\mathbf{Z}/p^d\mathbf{Z}$ and allows to consider the $\chi_\zeta$-equivariant sub-$F$-isocrystal, which is actually of rank~$1$. It is to be expected that the product of these $F$-isocrystals of rank~$1$ are realised by the original equation~\eqref{diffeq}. Our constructions would provide a computationally accessible exhibition of this correspondence.  The details and precise computations for establishing such a fact require lengths in contextualising that should be offered in another article.

\end{remark}	

\paragraph{The \emph{index} $\chi$.}\label{sectionindex}
 We assume the solubility of~\eqref{diffeq}. Associated to~\eqref{diffeq} is its \emph{index}~$\chi$. 
\subparagraph{} It is the index of the differential operator\footnote{Equivalently, the index of~$d:\Odagger\cdot e(T) \xrightarrow{f\mapsto df}\Odagger\cdot e(T)dT$.
}
 in the de Rham complex:
\begin{equation}\label{deRham}
0\xrightarrow{}\Odagger\xrightarrow{f(T)\mapsto df - L\cdot f\cdot dT} \Odagger\xrightarrow{}0
\end{equation}
namely, the Euler-Poincaré characteristic~$\chi=\dim H^0 - \dim H^1$ of the cohomology groups of the complex~\eqref{deRham}.

\subparagraph{The slope (\cite{Christol}[Definition~2.6 onward]).} There is yet another interpretation of this invariant, due to Robba\footnote{See \cite[§1.3 with a=1]{RobbaGAU93} for a statement with the algebra~$\mathcal{H}^+$ of functions converging on the closed disk instead of the dagger  algebra, and for the non solvable case. (see also~\cite[Theorem~10.2.2]{ChristolRobba}) The proper reference, that we lack, is left to the knowledgeable reader.

}. See~\cite{Christol}, \cite[§9-10]{ChristolRobba} for the notion of a Dwork (or Berkovich) generic point~$g_r$ at radius~$r\in\mathbf{R}_{\geq0}$ and more details. We consider the the radius of convergence~$\RoC(r)$ of~\eqref{diffeq} centered at~$g_r$: considering the differential equation given by the coefficient~$L(T-g_r)$, it is the radius of convergence of the solution~$\exp\left(P(T-g_r)-P(g_r)\right)$. Consider the function $\RoC$ viewed in logarithmic abscissa and ordinate:
\begin{equation}\label{RoC}
v\mapsto\log(\RoC(\exp(v))),
\end{equation}
It happens to be continuous and affine by part:  a polygonal line. Its right derivative at~$v=0$ (the right\footnote{If we parameter the axes with respect to valuations (opposite of~$\log_p$), this is the left slope. Compare the examples from the algorithm joined with~\cite{piexp1}.} slope) is~$\chi$.

\paragraph{Formula for the index.} By Proposition~\ref{equivalence}, a soluble~\eqref{diffeq} is characterised by the associated~\eqref{eHat}.
One should be able to recover the index from~\eqref{eHat}. For that purpose, we introduce the notation~$v_T(-)$ for the valuation associated with~$T$.
Namely
\begin{equation}\label{vt}
\text{~$v_T(\widehat{e}(T)-1)$ is the multiplicity of~$0$ as a root of~$\widehat{e}(T)-1$,   }
\end{equation}
obviating the case~$\widehat{e}(T)-1=0$ which corresponds to trivial equations. 
The algorithm joined to~\cite{piexp1} computes~$\tilde{e}(T)$ with some precision. With the slightest extra cost, this allows to exactly deduce~$\widehat{e}(T)$, then~$v_T(\widehat{e}(T))$ and finally~\eqref{mainformulap}.
\subparagraph{$p$-typical case.}	We first treat the $p$-typical case. This is the case where
\begin{equation}\label{Hypothesep}
 P(T)\in \bigoplus_{i\geq1} K\cdot T^{p^i}.
 \end{equation}

\begin{theorem}[Closed formula for the index in the $p$-typical case]\label{mainthmp}
Assume solubility and non triviality\footnote{The non-triviality assumptions becomes 
superfluous under the convention that for the null truncated series~$v_T(0\pmod{T^{D+1}})=+\infty$, so that
the fraction in~\eqref{mainformulap} evaluates to~$0$.}
 of~\eqref{diffeq} and assume~\eqref{Hypothesep}. 
Then the index of~\eqref{diffeq} is
\begin{equation}\label{mainformulap}
\chi=1-\frac{p^d}{v_T(\widehat{e}-1)}.\end{equation}
\end{theorem}
\begin{remark} As an implied statement:~$v_T(\widehat{e}-1)$ is a power of~$p$.
 As a corollary:~$\chi-1$ is the negative of a power of~$p$.
\end{remark}
\begin{example} In the three cases of~Remark~\ref{examples}, formula~\eqref{mainformulap}
becomes respectively~$\chi=1-p^0/1$ for~\eqref{example1};~$\chi=1-p^0/1$ for~\eqref{example2}; and~$\chi=1-p^1/p$ for~\eqref{example3}.

\end{example}

\begin{subequations}
\subparagraph{$p$-typical decomposition.}\label{ptypdecomp} In general we may uniquely write
\begin{equation}\label{ptypP}
P(T) = \sum_{1\leq m\leq D,\\ p\nmid m} P_m(T^m),\text{ where each }P_m\text{ satisfies~\eqref{Hypothesep}}.
\end{equation}
We can correspondingly construct~$e_m(T):=\exp(P_m(T))$ so that
\begin{equation} e(T)= \prod_{1\leq m\leq D,\\ p\nmid m} e_m(T^m). \end{equation}
We form the corresponding~$\tilde{e}_m(T)$ and still have
\begin{equation} \tilde{e}(T)= \prod_{1\leq m\leq D,\\ p\nmid m} \tilde{e}_m(T^m). \end{equation}
It happens~$\tilde{e}(T)$ is integral if and only if each factor~$\tilde{e}_m(T^m)$ is (cf.~\cite[§X]{piexp1}). Assuming solubility
for~\eqref{diffeq}, we can consider the corresponding~$\widehat{e}_m(T^m)$, and still have
\begin{equation}
 \widehat{e}(T)= \prod_{1\leq m\leq D,\\ p\nmid m} \widehat{e}_m(T^m). \label{ptyphat}
 \end{equation}
\end{subequations}

\subparagraph{General “global” case.} Recall~\eqref{ptyphat}, \eqref{indices}, \eqref{vt}.
\begin{theorem}[Closed formula for the index]\label{mainthm}
Assume solubility of~\eqref{diffeq} and consider the decomposition~\eqref{ptyphat}.
 Then the index of~\eqref{diffeq} is 
\begin{equation}\label{mainformula}
\chi=1-
\max_{1\leq m\leq D,\\ p\nmid m}
\frac{m\cdot p^{d_m}}{v_T(\widehat{e}_m-1)}
.\end{equation}
\end{theorem}

%

\section{Applications}

\subsection{Application to exponential sums.}\label{appliexpsums}
Let us indicate a last interpretation of the index. Given a $p$-typical soluble~\eqref{diffeq} one can construct families of exponential
sums, which are written down in~\cite{Morofushi}, and some generating function, the~$L$-function, actually an Euler factor, concretely
a polynomial of some degree, say~$\Delta$. The thesis~\cite{Morofushi} investigates these~$L$-functions. The object of its first, of two, part is this
degree~$\Delta$, and~\cite{Morofushi} succeeds in providing bounds by direct computations, for some low values of~$d$. 

On the other hand, the trace formula gives a cohomological interpretation of 
the $L$-function, and its degree is~$\dim(H^1)$. It follows~$\Delta$ is given:
\begin{itemize}
\item in the trivial case, by~$\dim(H^1)=0$ (but $\chi=\dim(H^0)=1\neq0$);
\item in the non trivial (but still soluble) case, by~$\dim(H^1)=-\chi$.
\end{itemize}
As a consequence, Theorem~\ref{mainthmp} answers\footnote{ Provided one has 
established the link between Pulita’s $\pi$-exponentials and the exponential sums written down in~\cite{Morofushi}. This
link is claimed without proof in~\cite{Morofushi}, and is the subject of a projected article in the series started by~\cite{piexp1}
and continued here.} the first problem studied in~\cite{Morofushi} with the closed formula~\eqref{mainformulap}.
Finally the algorithm accompanying~\cite{piexp1} allows to compute the right hand side of~\eqref{mainformulap}.

\subsection{Application to comparison.}\label{applicomparison}

\paragraph{Comparison map.}
  Let us consider the inclusion of de Rham complexes, into~\eqref{deRham}, of
\begin{equation}\label{deRhamAlg}
0\xrightarrow{}K[T]\xrightarrow{f(T)\mapsto df - L\cdot f\cdot dT} K[T]\xrightarrow{}0.
\end{equation}
The cohomology groups of~\eqref{deRham} are refered sometimes as \emph{analytic} cohomology, \emph{Dwork} cohomology or \emph{rigid} cohomology, etc.
For~\eqref{deRhamAlg}, one sometimes speak of \emph{rational} or \emph{algebraic} cohomology. We will use~\cite[§2.4]{Robba} terminology: Dwork and rational cohomology. The 
inclusion of complexes induces a comparison map from rational cohomology to Dwork cohomology, say
\begin{equation}\label{comparisonmap}
 H^0\xrightarrow{\mathrm{comparison_0}}H^0_\mathrm{rat}\qquad H^1\xrightarrow{\mathrm{comparison_1}}H^1_\mathrm{rat}
\end{equation}
easily seen to be injective and surjective respectively (loc.~cit.).

A recurrent difficulty has been that it is not always an isomorphism: For example Boyarsky principle, on variation of cohomology and Gross-Koblitz formula
for the $p$-adic Gamma function, relies on an interplay between the two
cohomology spaces:
\begin{itemize}
\item  a Frobenius endomorphism which comes from the Dwork cohomology,
\item  a functional equation
which comes from the rational cohomology.
\end{itemize}
\paragraph{Comparison criterion.}
By injectivity and surjectivity property, the fact that the comparison map is an isomorphism is equivalent to the identity of the dimensions of the Dwork and rational cohomology groups. We precisely computed it for the Dwork cohomology. For the rational cohomology,
this is simply given by the degree of~$L(T)$. We therefore can state the following.
\begin{corollary}\label{corocomparison}
 Consider a solvable~\eqref{diffeq}, choose~$D=\deg(P)$ and let~$D=m\cdot p^n$ with~$p\nmid m$. The comparison map~\eqref{comparisonmap} from de Rham cohomology with coefficients in~$K[T]$ to cohomology of~\eqref{deRham} is an isomorphism if and only if (equivalently):
\begin{itemize}
\item one has~$\chi=1-D$;
\item the factor~$\widehat{e}_m(T)$ in~\eqref{ptyphat} has non zero derivative 
at~$0$.
\end{itemize}
\end{corollary}
Let us note that in the special case~$p\nmid D$, ie~$D=m$ this reduces to an innocuous check, namely that:
\begin{equation}\label{innocuous}
\text{the dominant coefficient~$a_D$ of~$P$ satisfies~$\abs{a_D}=\abs{\pi_0}$.}
\end{equation}
\begin{example} As an illustration consider the following example. It is related the conjecture formulated in~\cite{Loeser}. 
Let~$P$ be a polynomial with coefficients in~$\QQ$,
and write~$D=\deg(P)$. If~$p$ is a large enough prime, then we may assume that every coefficient of~$P$ is a $p$-adic unit, 
as well as~$D!$. Consider such a~$p$, a corresponding~$\pi_0$, and the series~$e(T)=\exp(\pi_0\cdot P(T))$.
Then the~$p$-typical decomposition corresponds to the monomial decomposition of~$P$. For every monomial~$a_iT^i$,
the series~$e_i(T)=\exp(\pi_0 a_i T^i)$ is easily seen to have radius~$1$, and index~$1-i$. The check~\eqref{innocuous}
is satisfied for any~$e_i(T)$ and for~$e(T)$. We can conclude as follows.
\begin{proposition}
For all but finitely many~$p$, the series~$\exp(\pi_0\cdot P(T))$
defines a soluble differential equation with index~$1-D$ and for which the comparison of cohomologies~\eqref{comparisonmap} is an isomorphism.
\end{proposition}
For the remaining~$p$, the equation may be trivial, not solvable or lack comparison. For the first two issues, the computation
of the radius of convergence may help choose a suitable change of variable. For the lack of comparison, the procedure below may apply.

\end{example}

\paragraph{Factorisation}	
Without detailing the proof, we  mention a complement. Assume for convenience that~$\kappa=\mathbf{Z}/(p)$. If~$P=a_{1}T+a_{p}T^{p}+\ldots+a_{p^d}T^{p^d}$
 is a polynomial satisfying~\eqref{Hypothesep}, we define its \emph{shift}
as
$$VP= a_{p}T+a_{p^2}T^{p}+\ldots+a_{p^d}T^{p^{d-1}}.$$
\begin{proposition}[Complement to Corollary~\ref{corocomparison}]
Assume a solvable~\eqref{diffeq} does not provide comparison isomorphism. 
Define~$F=VP_m(T^m)-P_m(T^m)$. The decomposition
$$P=F+(P-F)$$
is such that
\begin{itemize}
\item the term~$F$ defines a trivial differential module;
\item the term~$(P-F)$ has degree~$\leq D-1$.
\end{itemize}
In other words, the equation defined by~$P-F$ is equivalent to the one defined by~$P$ but has strictly lower dimensional rational cohomology.
\end{proposition}
The only non elementary statement is the triviality of the differential module attached to~$F$. This (mostly) amounts to the existence of a Frobenius structure.

Applied iteratively, this procedure can restore the lack of comparison without changing the Dwork cohomology.
We dedicate to yet another future article the consequences of this application to Boyarsky principle.
For future reference, we call~$F$ the \emph{superfluous factor} of degree~$D$. 

\section{Demonstrations}
\paragraph{Products of differential equations and index.}\label{maxradius}
Recall that for two converging series with distinct radius of convergence, the radius of the product series 
is the smaller of the two. (distinctness is paramount here.) A variation of this observation, coupled with
the continuity of the function radius of convergence, implies the following.
\begin{lemma}[{\cite{Christol}, \cite[Corollary~2.4.8]{Pulita}}]\label{lemmamaxraxon}
 Given two polynomials~$L_1$ and~$L_2$ whose corresponding equations~\eqref{diffeq}
are solvable, but with distinct index~$\chi_1$ and~$\chi_2$, the product equation, with coefficient~$L=L_1+L_2$,
is still solvable and has index~$\min\{\chi_1;\chi_2\}$.
\end{lemma}
We refer to~\cite{Christol} for a detailed explanation, and for these other facts:
\begin{enumerate}
\item[(i)]\label{argument1} For any unit~$u$ in~$R^\times$ the equations given by~$L(T)$ and by~$L(u\cdot T)$ share the same index. (There is an obvious isomorphism of de Rham complexes.)
\item[(ii)]\label{argument2} For any~$m$ positive and prime to~$p$, and a soluble equation given by~$L(T)$ and of index~$\chi$, the derived equation given by~$L(T^m)$
is soluble and has index~$\chi^\prime$ such that~$\chi^\prime-1=m\cdot(\chi-1)$. (cf. \cite[Proposition 2.8]{Christol})
\end{enumerate}

\begin{proof}[Proof of Theorem~\ref{mainthm} from Theorem~\ref{mainthmp}] The decomposition~\ref{ptypdecomp} 
induces a corresponding decomposition of~\eqref{diffeq}. The equation corresponding to some~$P_m(T)$ is eligible for Theorem~\ref{mainthmp}: It is trivial or its index has the form~$1-p^i$ for some~$i$. By the fact~(ii) above, The equation corresponding to~$P_m(T^m)$ has index~$1-mp^i$.

It follows each of the \emph{non trivial} factors of~\eqref{diffeq} have distinct index. By Lemma~\ref{lemmamaxraxon} above, the index of~\eqref{diffeq}
is the minimum of the index of the factors. This yields~\eqref{mainformula} and concludes.

\end{proof}
\paragraph{Facts form Witt vectors theory.} For any (comutative unital) ring~$A$ recall the notation~$\Lambda(A)=1+TA[[T]]$, and that the Artin-Hasse series
\begin{equation}AH(T)=\exp(T+T^p/p+T^{p^2}/p^2+\ldots)\end{equation}
defines an element of~$\Lambda(A)$. We denote~$W(A)$ the ring of~($p$-typical) Witt vectors, and~$W_d(A)$ the ring of truncated Witt vectors of...
We identify~$W(A)$ with a subset~$\Lambda(A)$ through the Artin-Hasse map (which maps the unit of~$W(A)$ to the Artin-Hasse series).
It induces a embedding of~$W_d(A)$ into~$\Lambda(A)/(T^{D+1})$. Every element~$w$ in~$W(A)$ admits a unique decomposition
$$w=AH(w_0T)\cdot\ldots\cdot AH(w_iT^{p^i})\cdot\ldots,\text{ where }(w_0,\ldots)\in A^{\mathbf{Z}_{\geq0}},$$
and every element~$w$ in~$W_d(A)$ factors uniquely as
\begin{equation}\label{decompotruncated}
w\equiv AH(w_0T)\cdot\ldots\cdot AH(w_dT^{p^d})\pmod{(T^{D+1})},\text{ where }(w_0,\ldots,w_d)\in A^{d+1}.
\end{equation}
If~$A$ embeds in a~$\QQ$-algebra, so we can form the logarithm power series, $W(A)$ consists of series in~$\Lambda(A)$ whose 
logarithm falls into~$\prod_{i\geq 0} (\QQ\otimes A)\cdot T^{p^i}$. (compare~\eqref{Hypothesep})

\paragraph{} These finish the proof of our results.
\begin{proof}[Proof of Theorem~\ref{mainthmp}]
The property~\eqref{Hypothesep} on~$P(T)$ obviously extends to~$\tilde{P}(T)$. It means~$\tilde{e}(T)$ is a $p$-typical series. Equivalently,
the series~$\tilde{e}(T)\in\Lambda(R)/(T^{D+1})$ actually lies in~$W_d(R)$. The reduction~$\widehat{e}(T)$ in characteristic~$0$ lies in~$W_d(\kappa)$.
Then~$\widehat{e}(T)$ uniquely factors as
$$AH(w_0T)\cdot AH(w_1T^p)\cdot\ldots\cdot AH(w_dT^{p^d}).$$
Note that~$AH(T)\equiv T \pmod{(T^2)}$. It follows that
$$\log_p \left(	v_T(\hat{e}-1)\right)=\max\left\{0\leq i\leq d \middle| \forall 0\leq j<i, x_j=0 \right\}.$$
A factor~$AH(w_1T^p)$ comes from a trivial equation if~$w_i=0$, by~Theorem~\ref{thmarticle}, and else
comes from an equation of index~$-p^{d-i}$, by the lemma below.

Using argument~\ref{maxradius}, we conclude the proof.
\end{proof}

\begin{lemma} For any~$\lambda$ in~$\kappa^\times$ and any~$0\leq i\leq d$, there is a solvable differential equation~\eqref{diffeq} such that
\begin{equation}\label{lemmaformula}\widehat{e}(T)=AH(\lambda T^{p^i}),\end{equation}
and of index~$1-p^{d-i}$.
\end{lemma}
\begin{proof} Thanks to remark~(i) of §~\ref{maxradius}, we may assume~$\lambda=1$. Let us denote (the~``$\pi$-exponentials’’ of~\cite{Pulita}, cf.~\cite[§B]{piexp1})
$$e_k(T)=\exp\left(\pi_kT+\ldots+\pi_0T^{p^k}/{p^k}\right).$$
An equation satisfying~\eqref{lemmaformula} is the one such that
\begin{equation}\label{lemmaproof} e(T)=e_{d-i}(U) \text{ where }U=T^{p^i}.\end{equation}
Pulita proved that~$e(T)=e_{d-i}(T)$ defines an equation of index~$1-p^{d-i}$ (\cite{Pulita}).
It also proved this equation admits a Frobenius structure. It is then equivalent to the equation
defined by~\eqref{lemmaproof}.
\end{proof}

\begin{remark}\label{rk} 
The decomposition~\eqref{decompotruncated} is classical in Witt vectors theory, at least since Cartier. We applied
it to~$\widehat{e}(T)$. It applies equally to~$\tilde{e}(T)$. Its counterpart for the series~$e(T)$ itself is
a decomposition into~$\pi$-exponentials, and is due to Pulita. In a sense, we exhibit and retreive here Pulita’s
decomposition as a Cartier dual of decomposition~\eqref{decompotruncated}.
\end{remark}
\begin{remark}{Globalisation of Remark~\ref{rk}.}
Combining the~$p$-typical decomposition and $\pi$-exponential decomposition yield, for~$\widehat{e}(T)$ in~$\Lambda(\kappa)/(T^{D+1})$
a factorisation
$$\widehat{e}(T)=\prod_{1\leq n=mp^e\leq D} AH(u_nT^n)$$
which is unique\footnote{Unicity holds for any series~$AH(T)$ such that~$v_T(AH(T)-1)=1$. For the series~$1-T$ the~$u_i$
are the universal Witt vector coordinates. But for~\eqref{lastformula}, this is important to choose the~$p$-typical~$AH(T)$.}.
In terms of the nullity of the~$u_i$, we recover the index as
\begin{equation}\label{lastformula}
\chi=
1-\max
\left\{m\cdot p^\floor{\log_p(D/n)}
~\middle|~
1\leq n\leq D,~u_n\neq0,~n=mp^e,~p\nmid m
\right\}.
\end{equation}

\end{remark}

\begin{appendix}
\section{Polynomial complexity} Together with~\cite{piexp1} is an algorithm which computes~$\tilde{e}(T)$ form~$P(T)$. We discuss here
two points which were left untouched: the $p$-adic precision required; and the complexity. The main computational operation is~$\tilde{P}(T)\mapsto \tilde{e}(T)$.

\begin{remark} The following has benefited discussions with Jan Tuitman. 
\end{remark}

\subsection{Complexity}
Let us write~$\tilde{L}(T)=\sum_{i=0}^{D-1}c_iT^i$ the derivative of~$\tilde{P}(T)$. In order to compute~$\tilde{e}(T)$ we do solve
the differential equation~$y^\prime=\tilde{L}\cdot y$ in~$K(\zeta)[T]/(T^{D+1})$. 
Writing~$\tilde{e}(T)=\sum_{i=0}^{D}b_iT^i/i!$, one has~$b_0=1$ and the recurrence relation of order~$D$,
\begin{equation}\label{recurrence}
b_{i+1}= \sum_{k=0}^{D-1}  c_k \cdot b_{i-k}, \text{ (with $b_i=0$ for $i<0$)}
\end{equation}
which it will suffice to apply~$D-1$ times. This amounts to~$D-1$ summation of a total of the triangular~$D(D-1)/2$ number of
products, all to the required precision. This amounts to~$O(D^2)$ pairwise products and additions.

Assume
\begin{equation}
\text{precision is~$O(p^a)$ and ramification index is~$e$.}
\end{equation}
Assuming pairwise products and addition in a polynomial time~$O(({ae})^\eta)$,
this gets a complexity
\begin{equation}\label{complexity}
O(D^2({ae})^\eta).
\end{equation}
We usually have~$e=O(D)$ and~$a=O(D)$ (see below.) For a quasi-linear exponent~$\eta$, we get a quasi-quartic complexity. This is yet to multiply with the complexity of the residue field operations underlying our product and additions. (dependence in~$p$ and the residual degree).

\subsection{Precision} The truncated series~$\tilde{e}(T)$ has finitely many $p$-adic coefficients, all of which have infinitely many $p$-adic digits,
provided elements of the field~$K$ allows representation by digits.

\subsubsection{Ramification.}
We will assume that~$K$ is a finite extension of~$\mathbf{Q}_p$: it hence has finite residue field, of finite degree~$f_K$ over~$\mathbf{Z}/(p)$, and finite ramification index~$e_K$. The working field is the ramified cyclotomic extension~$K(\zeta)$, with same residue field, but may have ramification index~$e$ from~$p^d\cdot\frac{p-1}{p}\geq D\frac{p-1}{p^2}$ up to~$e_K\cdot p^d\cdot\frac{p-1}{p}\leq e_K\cdot D$. In order to distinguish between the complexity originating form~$K$ and from~$D$, we do not assume that~$\zeta$ belongs to~$K$.
\begin{equation}
e=\Omega(D) \text{ for a given $K$ and~$p$, and } e=O(D)\text{ for a given~$K$.}
\end{equation}
\subsubsection{Wanted precision.}
We still need to decide up to which precision we want to compute the coefficients of~$\tilde{e}(T)$. We want enough precision to determine the radius of convergence through formula~\cite[Théorème~3 (14)]{piexp1}, and, in the soluble case, for computing~$\widehat{e}(T)$. We will ask for enough precision 
in order to compute the first digit of the coefficients of~$\tilde{e}(T)$ achieving the maximum in the radius formula. In the solvable case, these are the coefficients reducing to non zero coefficients of~$\widehat{e}(T)$.

\subsubsection{Preparation.} It is easy to obtain the smallest~$k$ such that~$\tilde{P}({\pi_d}^{k}T)$ has integer coefficients. In terms of the normalised $p$-adic valuation~$v_p$,
$$ k = \ceil{\min v_p(\tilde{a}_i)/(i\cdot v_p(\pi_d))}\text{ where $\tilde{P}=\sum_1^D \tilde{a}_iT^i$.}$$
We will use the substitution of~$T$ by~${\pi_d}^{k}T.$ This way, by integrality of~$\tilde{P}$ the recurrence relations~\eqref{recurrence}
will always be computed in~$R$. Moreover, by the minimality condition, we obtain in the same time, the Gauß norm lower bound
\begin{equation}\label{minorationp} \Nm{P}\geq\abs{\pi_d}^{D-1}>\abs{p\pi_0}=\abs{p}^{p/(p-1)}.\end{equation}
Such substitution is likely to destroy the solvability property. Before reducing~$\tilde{e}(T)$ to compute~$\widehat{e}(T)$, we must not forget to substitute back the variable. Assuming this substitution, we will be able to express compute uniformly our need in precision in terms of absolute precision.

\subsubsection{Minoration}  Identifying~$\tilde{P}(T)$ with a truncated series in~$K(\zeta)[T]/(T^{D+1})$, we compute the transformation~$\exp: \tilde{P}(T)\mapsto \tilde{e}(T)$ from~$K(\zeta)[T]/(T^{D+1})$ to itself. This is a polynomial operation: we may substitute~$\exp$ with the truncated exponential
$$ 1+\tilde{P}+\tilde{P}^2/2+\ldots+\tilde{P}^D/D!,$$
whose truncation gives~$\tilde{e}(T)$.

We get back~$\tilde{P}(T)$ from~$\tilde{e}(T)$ by applying the truncated power series of~$\log(1-X)$ to~$1-\tilde{e}(T)$. Assuming~$\Nm{X}\leq\abs{\pi_0}$ we have~$\Nm{\log(1-X)}=\Nm{X}$ for the untruncated power series, hence
$$\Nm{\tilde{P}}\leq\Nm{\log(1-\tilde{e})}=\Nm{\tilde{e}}.$$
Together with~\eqref{minorationp}, this yields 
\begin{equation}\label{minoratione}
\Nm{\tilde{e}}>\abs{p\pi_0}.
\end{equation}

\subsubsection{Minimal radius} We seek to apply the formula for the radius. We know that at least one coefficient of~$\tilde{e}$ 
is at least~$\abs{p\pi_0}$ in absolute value. In the least favourable case, this is the coefficient of degree~$1$, and we need to 
compute the coefficient of degree~$D$ up to precision~$O((p\pi_0)^D)$ in order to use~\cite[Théorème~3]{piexp1}. Finally let 
us note that the coefficients of~$\tilde{e}$ are not the~$c_i$ from~\eqref{recurrence} but are the~$c_i/i!$. This involves an extra
$\abs{i!}$ factor in precision for computing~$c_i$. We recall~$1/\abs{D!}\leq1/\abs{\pi_0}^D$.
 In the end, an absolute precision~$O(p^a)$ is sufficient, with
$$a=D\cdot (1+2v_p(\pi_0))=D\cdot \frac{p+1}{p-1}=O(D).$$

\end{appendix}

\bibliographystyle{alpha}
\bibliography{crasswitt}

\frenchspacing





\end{document}